\documentclass[reqno,centertags, 12pt]{amsart}
\usepackage{amsmath,amsthm,amscd,amssymb}
\usepackage{esint}  %from Maxim for circular integral
\usepackage{latexsym}
\usepackage{graphicx}
\usepackage{enumitem}
\usepackage[mathscr]{eucal}
\usepackage[colorinlistoftodos,prependcaption,textsize=small]{todonotes}
\usepackage{xcolor}

\usepackage{hyperref}
\newcommand{\bbC}{{\mathbb{C}}}

\newcommand{\bbE}{{\mathbb{E}}}
\newcommand{\bbF}{{\mathbb{F}}}
\newcommand{\bbG}{{\mathbb{G}}}
\newcommand{\bbH}{{\mathbb{H}}}
\newcommand{\bbK}{{\mathbb{K}}}
\newcommand{\bbL}{{\mathbb{L}}}
\newcommand{\bbP}{{\mathbb{P}}}

\newcommand{\bbR}{{\mathbb{R}}}
\newcommand{\bbS}{{\mathbb{S}}}

\newcommand{\bbZ}{{\mathbb{Z}}}

\newcommand{\calC}{{\mathcal{C}}}

\newcommand{\calF}{{\mathcal F}}
\newcommand{\calG}{{\mathcal G}}
\newcommand{\calH}{{\mathcal H}}

\newcommand{\calK}{{\mathcal K}}
\newcommand{\calL}{{\mathcal L}}

\newcommand{\calQ}{{\mathcal Q}}

\newcommand{\calT}{{\mathcal T}}

\newcommand{\bdone}{{\boldsymbol{1}}}

\newcommand{\lb}{\label}

\newcommand{\ti}{\tilde  }
\newcommand{\wti}{\widetilde  }

\newcommand{\tr}{\text{\rm{Tr}}}

\newcommand{\bi}{\bibitem}

\newcommand{\beq}{\begin{equation}}
\newcommand{\eeq}{\end{equation}}
\newcommand{\ba}{\begin{align}}
\newcommand{\ea}{\end{align}}

\newcounter{smalllist}

\newcommand{\comm}[1]{}
\allowdisplaybreaks
\numberwithin{equation}{section}
\newtheorem{theorem}{Theorem}[section]
\newtheorem{proposition}[theorem]{Proposition}

\theoremstyle{definition}
\newtheorem*{remark}{Remark}

\newcommand{\jap}[1]{\langle #1 \rangle}
\newcommand{\norm}[1]{\lVert#1\rVert}

\begin{document}

\title[Periodic BC Trees]{Periodic Boundary Conditions for Periodic Jacobi Matrices on Trees}
\author[N.~Avni, J.~Breuer, G.~Kalai and B.~Simon]{Nir Avni$^{1,5}$, Jonathan Breuer$^{2,6}$, Gil Kalai$^{3,7}$ \\and Barry Simon$^{4,8}$}

\thanks{$^1$ Department of Mathematics, Northwestern University, Evanston, IL  E-mail: avni.nir@gmail.com}

\thanks{$^2$ Institute of Mathematics, The Hebrew University of Jerusalem, Jerusalem, 91904, Israel. E-mail: jbreuer@math.huji.ac.il.}

\thanks{$^3$ Institute of Mathematics, The Hebrew University of Jerusalem, Jerusalem, 91904, Israel. E-mail: gil.kalai@gmail.com}

\thanks{$^4$ Departments of Mathematics and Physics, Mathematics 253-37, California Institute of Technology, Pasadena, CA 91125, USA. E-mail: bsimon@caltech.edu.}

\thanks{$^5$ Research supported in part by NSF grant DMS-1902041.}

\thanks{$^6$ Research supported in part by Israel Science Foundation Grant No. 399/16}

\thanks{$^7$ Research supported in part by ERC advanced grant 834735}

\thanks{$^8$ Research supported in part by NSF grant DMS-1665526.}

\dedicatory{In celebration of the 100th birthday of Shmuel Agmon.}

\

\date{\today}

\begin{abstract}  We consider matrices on infinite trees which are universal covers of Jacobi matrices on finite graphs.  We are interested in the question of the existence of sequences of finite covers whose normalized eigenvalue counting measures converge to the density of states of the operator on the infinite tree.  We first of all construct a simple example where this convergence fails and then discuss two ways of constructing the required sequences: with random boundary conditions and through normal subgroups.
\end{abstract}

\maketitle

%%%%%%%%%%%%%%%%%%%%%%%%%%%%%%%%%%%%%%%%%%%%%%%%%%%%%%%%%%%%%%
\section{Introduction} \lb{s1}
%%%%%%%%%%%%%%%%%%%%%%%%%%%%%%%%%%%%%%%%%%%%%%%%%%%%%%%%%%%%%%

This paper is a contribution to the growing literature on periodic Jacobi matrices on trees \cite{ABS, BGVM, CSZtree, GVK}.  In particular, we refer the reader to \cite{ABS} whose notation and ideas we will follow; see that paper for the definitions from graph theory that we will use.  We will always associate a graph with the obvious topological space.

One starts out with a finite leafless graph, $\calG$, and a Jacobi matrix on that graph.  By this we mean a matrix with indices labelled by the vertices in the graph, whose diagonal elements are a function, $b$, on the vertices and off diagonal elements, which are only non-zero for pairs of vertices which are the two ends of some edge, with matrix elements determined by a strictly positive function, $a$, on the edges.  The universal cover, $\calT$, of $\calG$ is always a leafless tree.  There is a unique and natural lift of any Jacobi matrix, $J$, on $\calT$ to an operator, $H$, (still called a Jacobi matrix) on $\ell^2(\calT)$.  Because $H$ is invariant under a group of deck transformations on $\calT$ (see \cite{Hatch} or \cite[Section 1.7]{BCA} for the covering space language we exploit), we call $H$ a \emph{periodic Jacobi matrix}.

There are three big general theorems in the subject: (1) a result of Sunada \cite{Sun} (see also \cite{ABS}) on labelling of gaps in the spectrum that implies the spectrum of $H$ is at most $p$ bands where $p$ is the number of vertices in the underlying finite graph $\calG$, (2) a result of Aomoto \cite{AomotoPoint} stating that if $G$ has a fixed degree, then $H$ has no point spectrum (see also \cite{ABS, BGVM}), and (3) a result of Avni--Breuer--Simon \cite{ABS} that there is no singular spectrum because matrix elements of the resolvent are algebraic functions.  Besides a very few additional theorems, the subject at this point is mainly some interesting examples and lots of conjectures and questions.

A basic object is the \emph{density of states} (DOS), $dk$ (and the weight this measure assigns to $(-\infty,E)$, the \emph{integrated density of states} (IDS), $k(E)$). We fix a finite graph, $\calG$, with $p$ vertices and $q$ edges.  For each vertex, $j\in\calG$, the spectral measure for $H$ at vertex $r\in\calT$, $d\mu_ r$ is the same for all $r\in\calT$ with $\Xi(r)=j$, where $\Xi:\calT\to\calG$ is the covering map.  The DOS is defined by picking one $d\mu_r$ for each $j\in\calG$, summing over $j$ and dividing by $p$, the number of vertices in $\calG$, that is
\begin{equation}\label{1.1}
  dk(\lambda) = \frac{1}{p}\sum_{j\in \calG; r \text{ so that } \Xi(r)=j} d\mu_r(\lambda)
\end{equation}

In the one dimensional case, a fundamental fact \cite{AS83} is that the DOS can be computed as a limit of normalized eigenvalue counting of operators restricted to boxes with either periodic or free boundary conditions (we say one dimensional because we are discussing trees but the result in \cite{AS83} holds also on $\bbZ^\nu$).  In \cite{ABS}, it is proven that for any tree but the line, this result fails for the free boundary condition case and our goal in this paper is to explore the case of periodic boundary conditions.

Of course, it isn't quite clear what one should mean by periodic BC which is illuminated by the lego block picture. Before discussing that, we want to remind the reader about what \cite{ABS} calls the lego block picture of periodic Jacobi matrices on trees, an unpublished realization of Christiansen, Simon and Zinchenko.  In this view, the fundamental trees are $\calT_{2\ell}$, the homogenous degree $2\ell$ tree.  Any leafless finite graph, $\calG$ has as its universal cover, a tree which can be realized as some $\calT_{2\ell}$ with each vertex replaced by a finite tree! (see also \cite{GilA3, GilA2} for discussions of universal covers of non-regular graphs).  Indeed, if $\calG$, has $p$ vertices and $q$ edges, there is a maximal connected subtree, $\calF$, obtained by removing $\ell$ edges and it is then easy to see that the homotopy group of $\calG$ is $\bbF_\ell$, the free non-abelian group on $\ell$ generators, whose natural Cayley graph is $\calT_{2\ell}$.  In general we think of the $2\ell$ edges coming out of each vertex as labelled $e_1^+,\dots,e_\ell^+,e_1^-,\dots,e_\ell^-$ with the rule that each $e_j^\pm$ has to be connected to an $e_j^\mp$ edge of a neighboring graph.

Now view $\calG$ as the finite tree, $\calF$, with $\ell$ edges added but rather than removing these edges, one imagines cutting them and leaving $2\ell$ half edges $e_1^+,\dots,e_\ell^+,e_1^-,\dots,e_\ell^-$ dangling.  These are the lego blocks which we place one at each vertex of $\calT_{2\ell}$  and connect the dangling edges by the above rules.  Thus the Hilbert space for $\calH=\ell^2(\calT)$ on which $H$ acts is replaced by vector valued functions on $\ell^2(\calT_{2\ell};\bbC^p)$.  $H$ is now a block Jacobi matrix where the diagonal elements are $p\times p$ block Jacobi matrices obtained by restricting the original Jacobi matrix on $\calG$ to the subgraph $\calF$, i.e.\ dropping the $a$'s from the cut edges.  The off diagonal piece linking two neighboring vertices $v_1$ and $v_2$ is a rank one $p\times p$ matrix, which associated to the edge $e_j$ linking $v_1$ and $v_2$ in $\calT_{2\ell}$ has a single non-zero matrix element (namely the $\calG$ Jacobi parameter associated to $e_j$) at the vertices inside the copies of $\calF$ corresponding to the vertices $v_1$ and $v_2$.

Because of the lego representation, if we prove convergence of eigenvalue counting measure for some class of periodic BC operators for the case of scalar $\calT_{2\ell}$ but where these periodic operators respect $e_1^+,\dots,e_\ell^+,e_1^-,\dots,e_\ell^-$ labelling (see below after \eqref{1.6}), then automatically we have results for infinite trees built over general finite leafless graphs $\calG$.  Thus, henceforth, we will only discuss and state results for the scalar $\calT_{2\ell}$ case bearing in mind that these automatically imply the more general results.

Because all the measures we consider (i.e., normalized eigenvalue counting and DOS) have supports in some fixed bounded set (once we fix all the Jacobi parameters), to prove convergence of measures, it suffices to prove convergence of moments.  In this regard, we will need a graphical representation of the moments of the density of states.  If $d\mu_j$ is the spectral measure of a point $j\in\calT$, then
\begin{equation}\label{1.2}
  \int \lambda^k d\mu_j(\lambda) = \jap{\delta_j,H^k\delta_j}
\end{equation}
Expanding $H^k$, one sees that
\begin{equation}\label{1.3}
  \jap{\delta_j,H^k\delta_j}=\sum_{\omega\in W_{j,k}} \rho(\omega)
\end{equation}
where $W_{j,k}$ is the set of all ``walks'' of length $k$ starting and ending at site $j$, i.e. $\omega_1,\dots,\omega_{k+1} \in\calT$ where $\omega_1=\omega_{k+1}=j$ and for $m=1,\dots,k$, one has that either $\omega_{m+1}=\omega_m$ or $\omega_m$ and $\omega_{m+1}$ are neighbors in $\calT$ (i.e. two ends of a single edge).  Moreover
\begin{equation}\label{1.4}
  \rho(\omega) = \rho_1(\omega)\dots\rho_k(\omega)
\end{equation}
\begin{equation}\label{1.5}
  \rho_m(\omega) = \left\{
                     \begin{array}{ll}
                       b_{\omega_m}, & \hbox{ if } \omega_m=\omega_{m+1} \\
                       a_{(\omega_m,\omega_{m+1})} & \hbox{ if } \omega_m\ne\omega_{m+1}
                     \end{array}
                   \right.
\end{equation}

On the other hand, suppose we have a finite cover, $\calG_r$ of $\calG$ (of covering order $r$) and we let $H_r$ be the lift of $J$ to $\calG_r$, $n_r=\#(\calG_r)$ and $N_r$ the normalized eigenvalue counting measure for $H_r$ (later we'll sometimes use $N^{(r)}$), then
\begin{align}
  \int \lambda^k dN_r(\lambda) &= n_r^{-1} \tr(H_r^k) \nonumber \\
                               &= n_r^{-1} \sum_{j\in\calG_r} \jap{\delta_j,H_r^k\delta_j}  \nonumber \\
                               &= n_r^{-1} \sum_{j\in\calG_r} \sum_{\omega\in W_{j,k,r}} \rho(\omega) \lb{1.6}
\end{align}
where $W_{j,k,r}$ is defined like $W_{j,k}$ except that we require $\omega_m\in\calG_r$ instead of $\omega_m\in\calT$.

In the tree, the only paths that start and end at $j$ retrace where they have been so a little thought shows that difference between the average of \eqref{1.3} over a cell and \eqref{1.6} is due to the existence of simple closed loops in $\calG_r$.  Thus if there are no short closed loops in $\calG_r$ as $r\to\infty$, we expect that $dN_r$ will converge to $dk$.  Indeed, it suffices that there be few such closed path compared to $n_r$.

If the operator is acting on $\ell^2(\calT_{2\ell};\bbC^p)$ with matrix Jacobi parameters, there is still a random walk representation.  In \eqref{1.5} the $a$'s and $b$'s are replaced by matrices, the order in \eqref{1.4} matters and the right side of \eqref{1.4} is $\tr(\rho_1(\omega)\dots\rho_k(\omega))$.  All the arguments in later sections where we prove results for general period $1$ operators on $\ell^2(\calT_{2\ell})$ easily extend to matrix valued period $1$ Jacobi operators on $\ell^2(\calT_{2\ell};\bbC^p)$.  Then via the Lego representation we get results for general scalar periodic Jacobi matrices on trees.

One way to construct natural periodic boundary condition objects is to start with a fixed vertex in $\calT_{2\ell}$, most naturally, the identity after $\calT_{2\ell}$ is identified with $\bbF_\ell$, and look at $\Lambda_r$, the set of all vertices a distance at most $r$ from the centered point.  The boundary of this set (i.e. all points a distance equal to $r$ from the center), $\partial\Lambda_r$ has $(2\ell)(2\ell-1)^{r-1}$ points.  Each of them can be viewed as having $2\ell-1$ dangling half edges sticking out which have natural labels with one of $e_1^+,\dots,e_\ell^+,e_1^-,\dots,e_\ell^-$ missing.  We can form a natural cover  of the basic $\calG$ (which for this case is a graph, $\wti{\calG}_\ell$, with one vertex and $\ell$ self loops) by pairing each dangling $e_j^+$ with some dangling $e_j^-$.

For a natural way to do this, view $\calT_{2\ell}$ as the Cayley graph of $\bbF_\ell$, take the center of $\Lambda_r$ to be the identity in $\bbF_\ell$ so that $\Lambda_r$ is all words in $e_1^+,\dots,e_\ell^+,e_1^-,\dots,e_\ell^-$ (with the only relation being that $e_j^-$ is the inverse of $e_j^+$) of length at most $r$.  If $b=w_1\dots w_r\in\partial\Lambda_r$, then $\tilde{b}=w_1^{-1}\dots w_r^{-1}$ are distinct points in  $\partial\Lambda_r$ whose dangling edges match (i.e. for each $e_j^\pm$ coming out of $b$ where is an $e_j^\mp$ coming out of $\tilde{b}$) and we can form a cover, $\calG_r$, by joining together these dangling edges.  This is in some way the most natural kind of periodic BC object and in Section \ref{s2} we will see this has too many short closed loops and its normalized eigenvalue counting measures do \emph{not} converge to $dk$.  Our main result in Section \ref{s2} will be that if the dangling edges in $\Lambda_r$ are connected in a random manner, then the eigenvalue counting measures do converge to $dk$.  This is because we'll show such random graphs have very few small loops.  This result is related, of course, to the result of McKay \cite{McK} that the eigenvalue measures for Laplacians of large random graphs of fixed degree converge to the DOS of the Laplacian of a tree of the same degree.

There is, of course, another way of thinking about collections of larger and larger finite covers of a fixed graph, $\calG$. By the theory of covering spaces \cite{Hatch} the covers, $\calC$, of $\calG$ are associated to subgroups of the fundamental group of $\calG$ and $\calC$ is a finite cover if and only the subgroup is of finite index.  We'll sketch the ideas here with details in Section \ref{s3}.  Thus a sequence of periodic BC objects is a sequence of finite index subgroups $H_1, H_2, \dots$ of $\pi_1(\calG)$.  The naive notion of their going to infinity is that for any finite subset, $F$, of $\pi_1(\calG)$ not containing the identity, eventually, one has that $H_j\cap F$ is empty.  One might expect that is the same as saying there are no small closed loops but that is, in general, wrong (as can be seen by looking at the example at the start of Section \ref{s2}): the point is that the condition
\begin{equation*}
    \text{(condition K)}\quad \forall_{\text{finite }F\text{ not containing the identity}} H_j\cap F=\emptyset, \text{ eventually}
\end{equation*}
only implies the absence of small closed loops containing the base point of the cover.  For other points we need $hH_jh^{-1}\cap F=\emptyset$ for conjugates of $H$.  Thus a special role is played by covers, which we'll call \emph{homogenous covers}, associated to normal subgroups of $\pi_1(\calG)$ (we note that the subgroups associated to the covers of Section \ref{s2} are (once $r$ is large) never normal).  If all the $H_j$ are normal, we will see easily that condition K is enough to imply that eigenvalue counting measures converge to the DOS.  So the issue becomes the existence of sequences of normal subgroups of $\bbF_\ell$ that obey condition K.  This is a well known folk theorem that we'll discuss in Section \ref{s3} providing a second class of periodic BC objects for which the desired convergence result holds.

In summary, we construct two classes of periodic BC Hamiltonians for which the normalized eigenvalue counting measures converge to the tree DOS.  In Section \ref{s2}, we look at $\Lambda_r$ with random pairings of the dangling ends.  In Section \ref{s3} we construct homogeneous covers with this property.

\begin{remark} In general (unless the tree is of degree two), the spectrum of the Jacobi matrix on the finite graph is not contained in the spectrum of its lift to the tree \cite{SS}. The convergence of the eigenvalue counting measures of random covers to the density of states says in particular that these covers have few eigenvalues outside the spectrum of the tree.  Related to this, using Bordenave-Collins \cite{BordC}, it is possible to show that for the sequence of random covers (chosen uniformly of the set of all degree k covers) of growing degree, the Hausdorff distance between the added eigenvalues and the spectrum of the tree converges in probability to zero, a direct analog of McKay's result \cite{McK}.  That is there is a third collection of covers for which we can prove convergence of the normalized eigenvalue counting measure to the DOS of the tree.
\end{remark}

Shmuel Agmon is a giant in spectral theory whose innovations have long charmed and benefitted us.  It is a great pleasure to present him this bouquet on his $100^{th}$ birthday.

%%%%%%%%%%%%%%%%%%%%%%%%%%%%%%%%%%%%%%%%%%%%%%%%%%%%%%%%%%%%%%
\section{Random Centered BC} \lb{s2}
%%%%%%%%%%%%%%%%%%%%%%%%%%%%%%%%%%%%%%%%%%%%%%%%%%%%%%%%%%%%%%

In this section, we will fix $\ell\in\bbZ; \ell\ge 2$ and take $\calG$ to be the graph with one vertex and $\ell$ self loops so that the covering tree, $\calT_{2\ell}$, is the homogeneous tree of degree $2\ell$, which we can identify with the Cayley graph of $\bbF_\ell$, the free group on $\ell$ generators. We also fix a Jacobi matrix, $J$, on $\calG$. We'll discuss covers of $\calG$ where we can label the edges coming out from a vertex as $e_1^+,\dots,e_\ell^+,e_1^-,\dots,e_\ell^-$ with the rule that each $e_j^\pm$ has to be connected to an $e_j^\mp$ edge of a neighboring vertex of the cover.  As we explained in the introduction, if we prove results about a class of finite covers whose normalized eigenvalue counting measure converges to the tree DOS, then using lego blocks, there are automatically results about general graphs which we will not state explicitly (essentially by extending the results to vector valued functions and block Jacobi matrices).

We will first describe our general framework, then give the example mentioned in the introduction with too many short loops, next state our general result and finally describe its proof.

For the time being, we fix $r$ although eventually, we will use $r$ as a label on sequences where we take $r\to\infty$.  In $\calT_{2\ell}$, let $\Lambda_r$ be the set of all vertices a distance at most $r$ from the origin of $\calT_{2\ell}$ thought of as $\bbF_\ell$ so that $\Lambda_r$ is the all words in $e_1^+,\dots,e_\ell^+,e_1^-,\dots,e_\ell^-$ of length at most $r$ (with the relations that $e_j^-=(e_j^+)^{-1}$).  Let $M_r=(2\ell-1)^r$  so that
\begin{equation}\label{2.1}
  \#(\partial\Lambda_r)=\tfrac{2\ell}{2\ell-1}M_r; \qquad \#(\Lambda_r)= \tfrac{\ell M_r-1}{\ell-1}
\end{equation}
are comparable to $M_r$. For each vertex in $\partial\Lambda_r$ we imagine extra half edges labelled by all the labels among $e_1^+,\dots,e_\ell^+,e_1^-,\dots,e_\ell^-$ except for the label of the edge that connects that vertex to the interior of $\Lambda_r$. It is easy to see that for each $j$ and $\pm$ there are exactly $M_r$ $e_j^\pm$ dangling half vertices.  By a \emph{pairing}, we mean a bijective association of each $e_j^+$ half-edge to a $e_j^-$ half-edge (at the same or a different vertex in $\partial\Lambda_r$).  The set of all pairings we denote by $\calQ_r$.  If we fix a pairing, all other pairings are related by a permutation of each of the $e_j^-$ vertices, so the number of points in $\calQ_r$ is $(M_r!)^\ell$.  For each $q\in\calQ_r$, we get a graph, $\calG_q$ by adding to $\Lambda_r$ edges along the linked pairs and this graph is a cover of $\calG$.  We let $J_q$ be the associated Jacobi matrix on $\calG_q$ and let $N^{(r)}_q$ be the associated normalized eigenvalue counting measure.

In Section \ref{s1}, we defined a map $b\mapsto\ti{b}$ of $\Lambda_r$ which also maps $\partial\Lambda_r$ to itself (the square of the map is $\bdone$ and it leaves no point of $\partial\Lambda_r$ fixed).  Moreover, the map is such that if $\alpha\in\partial\Lambda_r$ is linked to the interior by $e_j^\pm$ then $\ti{\alpha}$ is linked to the interior by $e_j^\mp$.  We can thus define a natural element, $q_0$ of $Q_r$ by pairing all the dangling edges from each $\alpha$ to the matching dangling edges of $\ti{\alpha}$.

\begin{proposition} \lb{P2.1} For any choice of Jacobi parameters $J$ on $\calG$, the limit $\lim_{r\to\infty} \int\lambda^2 dN^{(r)}_{q_0}(\lambda)$ exists and is strictly large than the second moment of the associated infinite tree DOS.
\end{proposition}

\begin{proof}  Let $b$ be the Jacobi parameter for the vertex in $\calG$ and $a_1,\dots,a_\ell$ for the $\ell$ edges.  Since the full tree is fully translation invariant, the DOS is the spectral measure associated to any vertex so, by \eqref{1.2} and \eqref{1.3}, we see that
\begin{equation}\label{2.2}
  \int\lambda^2 dk(\lambda) = b^2+2\sum_{m=1}^{\ell} a_m^2
\end{equation}
since paths of length two that start and end at a vertex on the tree either stay at that vertex for two steps (giving $b^2$) or go out and back in one of the $2\ell$ edges $e_m^\pm$ giving each $a_m^2$ twice.

For all vertices in $\calG_{q_0}$ other than those on the boundary, these are the only paths of length $2$ so the contribution to \eqref{1.6} from sites $j\notin\partial\Lambda_r$ is as given on the right of \eqref{2.2}.  But there are additional paths of length two that start at $j\in\partial\Lambda_r$ go to $\ti{j}$ by one link and go back via another.  One gets $M_r$ factors of each $a_m^2$ by going to $\ti{j}$ along $e_m^+$ and coming back via $e_m^-$ and an additional $M_r$ factor by going to $\ti{j}$ along $e_m^-$ and coming back via $e_m^+$.  There is a different number, $C_r$, of terms that count up the number of vertices $j$ which have a pair of links between $j$ and $\ti{j}$ with $a_1$ on one link and $a_2$ on the other (counting multiplicities).  We will not make $C_r$ explicit except to note that $c\equiv \lim_{r\to\infty} C_r/n_r$ exists and is non zero. We also note that by \eqref{2.1}, $\lim_{r\to\infty} M_r/n_r =(2\ell-1)/2\ell$.  We conclude from \eqref{1.6} that
\begin{equation}\label{2.3}
  \lim_{r\to\infty} \int\lambda^2 dN^{(r)}_{q_0}(\lambda) = \frac{2\ell-1}{\ell}\sum_{m=1}^{\ell} a_m^2+ c\sum_{m\ne n} a_m a_n + \int\lambda^2 dk(\lambda)
\end{equation}
proving the claim.
\end{proof}

Thus, because of small closed loops, we can't expect the periodic BC eigenvalue counting measure to converge to the DOS for all possible pairs and we turn to random pairings.  We put normalized counting measure, $\Xi_r$, on $\calQ_r$ with associated probability, $\bbP_r$ and expectation, $\bbE_r$.  We want to consider sequences $(q_r)_{(1\le r < \infty)}\in\prod_{r=1}^{\infty}\calQ_r\equiv \calQ$ and put the product measure $\bigotimes_{j=1}^\infty \Xi_r$ on $\calQ$, so the $q_r$ are independent and randomly distributed.  Here is the main result of this section:

\begin{theorem} \lb{T2.2} For almost every choice $(q_r)\in\calQ$, we have that the normalized eigenvalue counting measures, $dN^{(r)}$, converges weakly to the DOS, $dk$.
\end{theorem}

We will prove this by a sequence of steps:

\emph{Step 1.} Fix a positive integer $m$ and a site $\alpha\in\Lambda_r$.  Call $q\in\calQ_r$ \emph{$m, \alpha$-bad} if there is a simple closed loop in $\calG_{q}$ of length at most $2m$ all within distance $m$ (in $\calG_q$-distance) of $\alpha$. Let
\begin{equation}\label{2.3A}
  B_{m,\alpha} = \{q\,\mid\,q \text{ is }m,\alpha\text{-bad }\}
\end{equation}
We will prove that there is a constant $T_m$ so that for all $r$ and all $\alpha\in\Lambda_r$ we have that
\begin{equation}\label{2.4}
  \bbP_r(B_{m,\alpha}) \le T_m/M_r
\end{equation}

\emph{Step 2.} If $q\in\calQ_r$ is not $m, \alpha$-bad, then the set of $\beta\in\calG_q$ a distance at most $m$ from $\alpha$ is a truncated degree $2\ell$ tree centered at $\alpha$.  From this it follows that
\begin{equation}\label{2.5}
  \jap{\delta_\alpha,(J_q)^m\delta_\alpha} =  \int \lambda^m dk(\lambda)
\end{equation}

\emph{Step 3.} Prove that for each positive integer, $m$, there is a constant $U_m$ so that
\begin{equation}\label{2.6}
  \bbE_r\left(\left|\int \lambda^m dN^{(r)}_{q_r}(\lambda) - \int \lambda^m dk(\lambda)\right|\right) \le U_m/M_r
\end{equation}

\emph{Step 4.} Prove that for almost every choice $(q_r)\in\calQ$, we have that
\begin{equation}\label{2.7}
  \lim_{q\to\infty} \int \lambda^m dN^{(r)}_{q_r}(\lambda) =  \int \lambda^m dk(\lambda)
\end{equation}
and deduce the theorem from this.

\begin{proof} [Proof of Step 1] This is the most involved argument in the paper and is critical. We use the word ``finitely'' to indicate that a quantity only depends on $m$ but can be chosen independently of $r$. The idea is that for each $x\in\partial\Lambda_r$ in some bad graph, $\calG$, finitely far from $\alpha$ we can get swap a extra edge in a small loop with an edge from $x$ to get a new graph $\calK$.  One shows that the number of times a given $\calK$ is bounded (independently of $r$) while the number of choices for $x$ grows like $M_r$ which yields a bound like \eqref{2.4}.

We will fix $r, \alpha\in\Lambda_r,$ and $m$ and call the half-edge $2\ell$ possibilities, $e_j^\pm$, colors labelled by $\gamma$. We will suppose $r$ is fairly large (to be specified later). Given, $p,s\in\partial\Lambda_r$, a $\Lambda_r$-distance at most $2m$ from each other, we let
\begin{align}\label{2.7A}
  B_{p,s,\gamma} &= \{q\in\calQ_r\,\mid\,\text{there is a path of length at most } \nonumber\\
                 &\null\qquad\quad 2m\text{ from } p\text{ to } s\text{ leaving } p\text{ along half edge }\gamma\}
\end{align}
where $\gamma$ is one of the $2\ell-1$ half edges that does not link $p$ to the interior of $\Lambda_r$.

Since any simple closed path within distance $m$ of $\alpha$ in some $\calG_q$ must contain a link not in $\Lambda_r$, we can label such paths be the first and last times the path lies in $\partial\Lambda_r$ and the direction it leaves that first point and we see that
\begin{equation}\label{2.7B}
  B_{m,\alpha} \subset \bigcup_{p,s,\gamma} B_{p,s,\gamma}
\end{equation}
where $p,s$ are points in $\partial\Lambda_r$ a $\Lambda_r$-distance at most $2m$ of $\alpha$ (since these points are linked to $\alpha$ by a path entirely in $\Lambda_r$). An over count of the number of such points is $2\ell(2\ell-1)^{2m-1}$).  In \eqref{2.7B}, $\gamma$ is one of the $2\ell-1$ half edges that does not link $p$ to the interior of $\Lambda_r$.  Thus
\begin{equation}\label{2.7C}
  \text{the number of terms in }\eqref{2.7B} \le (2\ell-1)[2\ell(2\ell-1)^{2m-1}]^2
\end{equation}

We will be considering a variety of graphs, $\calG$ with degree at most $2\ell$ at each vertex.  $d_\calG$ will denote distance on that graph, i.e. shortest path between two vertices.  If $\calG$ and $\calH$ have the same vertices but $\calH$ has more edges, clearly
\begin{equation}\label{2.7D}
  d_\calH(x,y) \le d_\calG(x,y)
\end{equation}
Let $D_r=\#(\Lambda_r)$ given by \eqref{2.1}.  Then if $\calG$ has maximum degree $2\ell$, it is easy to see that the number of points a distance at most $r$ from a fixed $x\in\calG$ is bounded by $D_r$.

Now suppose $\calG\in B_{p,s,\gamma}$ and suppose that $x\in\partial\Lambda_r$ with $d_\calG(x,s)> 2m+1$ and so that $x$ is not linked to the interior by a vertex whose link to $x$ is $\gamma$. The distance condition implies that $x$ is not linked to $p$ by the vertex labelled $\gamma$.  So we let $y\ne x$ be the vertex linked to $p$ by the edge labelled $\gamma$ coming out of $p$ and let $z\ne p$ be the vertex whose edge labelled by gamma has $x$ at the other end.  Let $\ti{\calG}$ be $\calG$ with the gamma edges coming out of $p$ and $z$ removed (so if say $\gamma=e_1^+$, we have that $x$ and $y$ have dangling $e_1^-$ edges).  By the distance condition, the path length at most $2m$ from $p$ to $s$ does not include $x$ or $z$ and thus
\begin{equation}\label{2.7E}
  d_{\ti{\calG}}(y,s) \le 2m
\end{equation}

Let $\tau_x(\calG)$ be the graph obtained from $\ti{\calG}$ by linking $p$ to $x$ and $z$ to $y$ by the edges labelled $\gamma$ coming out of the first points (i.e $\tau_x(\calG)$ is obtained from $\calG$ by switching end points the gamma edges coming out of $p$ and $z$).  Suppose that $\calK=\tau_x(\calG)$ for some $\calG\in B_{p,s,\gamma}$ and some $x$.  The key observation is the bound on how many $\calG$'s there can be with $\tau_x(\calG)=\calK$.  Clearly, to recover $\calG$ from $\calK$, it suffices to know $p, x, z, y$.  $p$ is fixed and $x$ is the vertex linked to it in $\calK$ by the $\gamma$ edge.  If two $\calG$'s with $\tau_x(\calG)=\calK$ have the same $y$ they are equal.  By \eqref{2.7E} and \eqref{2.7D}, $d_\calK(y,s) \le 2m$ so we conclude
\begin{equation}\label{2.7F}
  \text{the number of }\calG\text{ with }\tau_x(\calG)=\calK \le D_{2m}
\end{equation}

Instead of counting numbers of possibilities it is simpler to divide by the number of points in $\calQ_r$, and note that for each $\calG\in B_{p,s,\gamma}$ we can form $\tau_x(\calG)$ for at least $M_r-D_{2m+1}$ points and see that (for any single graph $\calL$, we have that $\bbP_r(\calL)=1/\#(\calQ_r)$)
\begin{equation}\label{2.7G}
  (M_r-D_{2m+1})\bbP_r(\calG) \le \sum_{K\,\mid\,\exists_x\text{ with }\tau_x(\calG)=K} \bbP_r(\calK)
\end{equation}
If we sum over all $\calG\in B_{p,s,\gamma}$ and use \eqref{2.7F} and the fact that $\sum_\calK \bbP_r(\calK)=1$ to conclude that
\begin{equation}\label{2.8H}
  (M_r-D_{2m+1})\bbP_r(B_{p,s,\gamma}) \le D_{2m}
\end{equation}

If $r$ so large that $M_r \ge 2D_{2m+1}$ we conclude using \eqref{2.7B} and \eqref{2.7C} that for such $r$ we have that
\begin{equation}\label{2.7I}
  \bbP_r(B_{m,\alpha}) \le \frac{(2\ell-1)[2\ell(2\ell-1)^{2m-1}]^2 2 D_{2m}}{M_r}
\end{equation}
which proves \eqref{2.5} with
\begin{equation}\label{2.7J}
   U_m= \max((2\ell-1)[2\ell(2\ell-1)^{2m-1}]^2 2 D_{2m}, 2D_{2m+1})
\end{equation}
since then, when $r$ is small, the right side is bigger than 1.
\end{proof}

\begin{remark} There is an alternate proof of this step that we have that relies on McKay's result.  While it is somewhat shorter, we decided to use this result because it is self-contained and conceptually simple.
\end{remark}

\begin{proof} [Proof of Step 2] Suppose that $\calG_q$ has no simple closed loop of size no more than $2m$ all of whose sites are within distance $m$ of $\alpha$.  Since it is a cover, every vertex in $\calG_q$ has degree $2\ell$.  All of the $2\ell$ edges coming out of $\alpha$ must have second ends different from $\alpha$ (to avoid a closed loop of length $1$) and from each other (to avoid a closed loop of length $2$).  Each of the $2\ell$ vertices a distance one from $\alpha$ have $2\ell-1$ edges coming out besides the one linking to $\alpha$.  Those edges can't have a distance one vertex as their other end (to avoid loops of length $1$ or $3$) and must have all different second vertices (to avoid loops of length $2$ or $4$) so there are $2\ell(2\ell-1)$ vertices a distance $2$ from $\alpha$.  Repeating this shows that the vertices in $\calG_r$ a distance at most $m$ from $\alpha$ is exactly the truncated tree centered at $\alpha$ as claimed.

Since this implies all walks of length $m$ starting and $\alpha$ are the same as would be on an infinite tree starting at $\alpha$, we obtain \eqref{2.5} from \eqref{1.3}.
\end{proof}

\begin{proof} [Proof of Step 3] Fix $\alpha$.  Let $f(q)=\jap{\delta_\alpha,(J_q)^m\delta_\alpha} -  \int \lambda^m dk(\lambda)$.  It is easy to see that $\norm{J_q}\le |b|+2\sum_{s=1}^{\ell}a_s \equiv \Gamma$.  It follows (ignoring possible cancelling from the minus sign) that for all $q$, one has that $|f(q)|\le 2\Gamma^m$.  Since $f$ vanishes off the set on the left side of \eqref{2.4}, we see that
\begin{equation}\label{2.8}
  \bbE_r(|f|) \le \norm{f}_\infty \bbP_r(q\,\mid\,q \text{ is }m,\alpha\text{-bad }) \le 2\Gamma^m T_m/M_r
\end{equation}
Summing over $\alpha$ and dividing by the number of $\alpha$ yields \eqref{2.6} with $U_m=2\Gamma^m T_m$.
\end{proof}

\begin{proof} [Proof of Step 4] By \eqref{2.6} and Markov's inequality
\begin{equation}\label{2.9}
  \bbP_r\left(\left|\int \lambda^m dN^{(r)}_{q_r}(\lambda) - \int \lambda^m dk(\lambda)\right|\ge M_r^{-1/2}\right) \le U_m/M_r^{1/2}
\end{equation}

Since $M_r$ grows exponentially in $r$, we have that
\begin{equation}\label{2.10}
   \sum_{r=1}^{\infty}  \bbP_r\left(\left|\int \lambda^m dN^{(r)}_{q_r}(\lambda) - \int \lambda^m dk(\lambda)\right|\ge M_r^{-1/2}\right) < \infty
\end{equation}
Thus, by the Borel-Cantelli lemma \cite[Theorem 7.2.1]{RA}, for a.e. $q\in\calQ$, we have that eventually $\left|\int \lambda^m dN^{(r)}_{q_r}(\lambda) - \int \lambda^m dk(\lambda)\right|\le M_r^{-1/2}$ so that a.e., any given moment of $dN^{(r)}_{q_r}$ converges to that moment of $dk$.  It follows that a.e., we have convergence for all moments.  Since there is an apriori compact set in $\bbR$ that supports all the measures, the Weierstrass density theorem \cite[Theorem 2.4.1]{RA} implies weak convergence a.e.
\end{proof}

\begin{remark}  The relation between the spectrum of Laplacians of finite graphs and that of the covering tree is a central theme in graph theory and is related to the notions of expander graphs, and Ramanujan graphs. Analysis of sizes of cycles for random graphs has a crucial role in proving spectral properties for random regular graphs and especially expansion properties and relatives of the Ramanujan properties (see Hoori, Linial, Wigderson's survey of expander graphs \cite{GilA2} and Friedman \cite{GilA1}). It could be of interest to extend notions related to expander graphs based on the Laplacian to Jacobi operators and various Schr\"{o}dinger operators. In particular it would be interesting to extend  to more general operators the Alon-Bopanna theorem (see \cite{GilA4} or \cite[Section 5.2]{GilA2}) that asserts that the spectrum of regular graphs is controlled by the spectrum of the covering tree.
\end{remark}

%%%%%%%%%%%%%%%%%%%%%%%%%%%%%%%%%%%%%%%%%%%%%%%%%%%%%%%%%%%%%%
\section{Homogeneous BC} \lb{s3}
%%%%%%%%%%%%%%%%%%%%%%%%%%%%%%%%%%%%%%%%%%%%%%%%%%%%%%%%%%%%%%

In this section, we will identify finite covers of the basic graph $\calG$ with $1$ vertex and $\ell$ self loops with subgroups of $\bbF_\ell$ and use this to find sequences of finite covers whose normalized eigenvalue counting measures converge to the DOS (distinct from the examples in Section \ref{s2}).  By using the lego block representation one could then construct such sequences of finite covers for the Jacobi matrix on any finite leafless graph.

We quickly recall some basics of covering spaces for this situation.  We can view the universal cover as the Cayley graph $\bbF_\ell$.  If $\calC$ is any cover of $\calG$, then there is a covering map $\pi:\bbF_\ell \to \calC$ and if $\bdone$ is the identity in $\bbF_\ell$ and we let $v_0=\pi(\bdone)$, then $H\equiv \pi^{-1}[v_0]$ is precisely those $h\in\bbF_\ell$ so that the simple path from $\bdone$ to $h$ is pushed by $\pi$ to a closed curve.  This easily implies that $H$ is a group and that $\pi$ from $H$ to the just mentioned closed path is an isomorphism of $H$ to the fundamental group, $\pi_1(\calC,v_0)$.  Thus there is a 1-1 correspondence between subgroups of $\bbF_\ell$ and equivalence classes of covers of $\calG$.  The finite covers correspond precisely to subgroups, $H$, of finite index.

We are interested in sequences of finite covers, $\{\calC_n\}_{n=1}^\infty$, which in some sense converge to $\bbF_\ell$.  One notion of this that the corresponding subgroups $H_n$ have the property that any given $h\ne\bdone$ lies in only finitely many $H_n$.  We'll call the corresponding covering Jacobi matrices \emph{a sequence of periodic BC objects converging to infinity}.  We are especially interested in the case where each $\calC_{n+1}$ is a finite cover of $\calC_n$ for all $n$, which we call \emph{a tower of periodic BC objects converging to infinity} (if it indeed converges to infinity). We have a tower if and only if the sequence $H_n$ is nested, i.e. $H_n\subset H_{n+1}$ in which case convergence to infinity is equivalent to  $\cap H_n = \{\bdone\}$.

As we've seen, to get convergence of the normalized eigenvalue counting measures to the DOS we need very few small loops.  At first sight, one might think that is automatic when the covers converge to infinity since $H$ is connected to closed loops through $v_0$ so a condition like $\cap H_n = \{\bdone\}$.  Indeed, by the random walk representation, this implies that $\jap{\delta_{v_0}, J_{\calC_n}^m\delta_{v_0}}$ converges to the $m$th moment of the DOS.  But, of course, it can't always be true for the normalized eigenvalue counting measures since all the periodic BC operators we discussed in Section \ref{s2} converge to infinity and by Proposition \ref{P2.1} we do not have convergence of the normalized counting measures for the sequence of $q_0$'s!

A little thought shows that if $v_1=\pi(g)$ for some $g\in\bbF_\ell$ and $\calC$ corresponds to the subgroup, $H$, then the paths through $v_1$ moved to $\bdone$ are given by $gHg^{-1}$. In terms of measuring how big the closed paths are, that might seem harmless, but we note that if, for example, $\ell=2$ and $a, b$ are generators of $\bbF_2$, then $(ab)^na(ab)^{-n}=h$ is distance $2n+1$ from $\bdone$, while if $g=(ab)^{-n}$, then $ghg^{-1}=a$ is close to $\bdone$.  The moral is that the absence of small loops through $\bdone$ does not imply the same for all points.  But, if $gHg^{-1}=H$, then no small loops through $v_0$ implies no small loops through $v_1$!  This suggests that finite covers associated to normal subgroups should be especially interesting.  We will call the cover associated to such a normal subgroup a \emph{homogeneous cover} for reasons that will become clear in a moment.

If $\calC$ is a homogeneous cover associated to  a normal subgroup $H$, then $\calC$ is the Cayley graph of the quotient group $G/H$.  In particular by group multiplication $G/H$ acts freely and transitively on $\calC$ by an action which preserves the Cayley links (so $\calC$ looks the same from any point which is why we call it homogeneous).  In particular, $\jap{\delta_{v}, J_{\calC_n}^m\delta_{v}}$ is independent of $v$, so we conclude that

\begin{theorem} \lb{T3.1} If $J_n$ a sequence of periodic BC objects converging to infinity which are all operators on homogenous covers, then its normalized eigenvalue counting measures converge to the DOS.
\end{theorem}

Of course, for this to be interesting, there have to exist such sequences so the following is interesting

\begin{theorem} \lb{T3.2} For any $\ell$ there exist nested sequences of finite index normal subgroups of $\bbF_\ell$ and so towers of homogenous covers converging to infinity.  In particular, for these towers, we get associated Jacobi matrices whose normalized eigenvalue counting measures converge to the DOS.
\end{theorem}

There is of course a huge literature and knowledge about the structure of $\bbF_\ell$ and we believe many experts would regard this theorem as folk wisdom, but we feel it is useful to sketch one explicit construction.  We note that if $\bbH_n\subset\bbG$ is a nested sequence of finite index normal subgroups with $\cap_n\bbH_n = \{\bdone\}$ and if $\bbK\subset\bbG$ is any subgroup of $\bbG$, then $\bbK\cap\bbH_n$ is a nested sequence of normal subgroups of $\bbK$ with $\cap_n(\bbK\cap\bbH_n) = \{\bdone\}$.  The strategy will be to show that for each $\ell$, $\bbF_\ell$ is isomorphic to a subgroup of $\bbS\bbL(2,\bbZ)$ and then to show that $\bbS\bbL(2,\bbZ)$ has the required family of nested normal subgroups.  So our proof is via a sequence of simple Propositions.

\begin{proposition} \lb{P3.3} For each $\ell>2$, $\bbF_\ell$ is isomorphic to a subgroup of $\bbF_2$.
\end{proposition}

\begin{proof} It is remarkable that this algebraic fact about discrete groups will be proven using covering space theory, a subject that arose in complex analysis! For $n\ge 2$, form a graph $\calK_n$ whose vertices are the points in $\bbZ_n=\bbZ/n\bbZ$, the integers mod $n$, and where $m$ is connected by four edges to the points $m\pm 1,m\pm 2$ (for $n\le 4$, one has to describe things more carefully and describe in terms of self-loops for $n=2$ and multiple edges for $n=2,3,4$ but it is still a degree $4$ graph with $n$ points).  A maximal spanning tree obviously has $n-1$ edges so one needs to remove $2n-(n-1)$ edges to get from $\calK_n$ to the tree and thus, the fundamental group of $\calK_n$ is $\bbF_{n+1}$.

On the other hand, as a homogenous degree $4$ graph, $\calK_n$ is a finite cover of $\calG_2$, the graph with a single point and two self loops whose fundamental group is $\bbF_2$ (for example one can get an explicit covering map by taking all vertices of $\calK_n$ to the single vertex of $\calG_2$ and taking the $m,m\pm 1$ edges to one self loop and the $m,m\pm 2$ edges to the other self loop with the $\pm$ edges going in opposite directions).  Taking $n=\ell-1$, the covering map induces an injection on fundamental groups realizing $\bbF_\ell$ as a subgroup of $\bbF_2$, indeed one of index $\ell-1$.
\end{proof}

\begin{remark} One can take any homogenous graph of degree $4$.  For example, if $\calH_2$ is the graph with $2$ vertices and $4$ edges between them, the spanning tree is a single edge and $\bbF_3$ is generated by loops that go from the base point to the other by the spanning tree and return by one of the three others.  The four edges map to $a^{\pm 1}, b^{\pm 1}$ under the induced covering map to $\calG_2$, so we see that $a^2,ab,ab^{-1}$ are free in $\bbF_2$ giving an explicit set of generators.
\end{remark}

\begin{proposition} \lb{P3.4} Each $\bbF_\ell$, $\ell\ge 2$ is isomorphic to a subgroup of $\bbS\bbL(2,\bbZ)$
\end{proposition}

\begin{proof}  By the previous proposition, it suffices to prove it for $\ell=2$.  It is well known that the two matrices
$\left(
  \begin{array}{cc}
    1& 2 \\
    0 & 1\\
  \end{array}
\right)$ and
$\left(
  \begin{array}{cc}
    1& 0 \\
    2 & 1\\
  \end{array}
\right)$
are free generators (and generate the group of matrices of the form
$\left(
  \begin{array}{cc}
    4a+1& 2b \\
    2c & 4d+1\\
  \end{array}
\right)$ for arbitrary $a,b,c,d\in\bbZ$).  This group is often called the Sanov group after \cite{Sanov} who first proved the generators are free.  For a simple proof in English, see Goldbeg-Newman \cite{GN}
\end{proof}

The proof of Theorem \ref{T3.2} is clearly completed by

\begin{proposition} \lb{P3.5} $\bbS\bbL(2,\bbZ)$ contains nested sequences of normal subgroups that converge to infinity.
\end{proposition}

\begin{proof} Let $H_n$ be the subgroup of all matrices in $\bbS\bbL(2,\bbZ)$ which are congruent to $\bdone$ mod $2^n$.  This is clearly a decreasing sequence of subgroups whose intersection is $\{\bdone\}$.  Since the set of matrices congruent to zero mod $2^n$ is obviously an ideal in $\bbS\bbL(2,\bbZ)$, if $C\equiv\mathbf{0}, \mod 2^n$ and $B\in\bbS\bbL(2,\bbZ)$, then $B(\bdone+C)B^{-1}-\bdone\equiv\mathbf{0} \mod 2^n$, so $H_n$ is normal.
\end{proof}

%%%%%%%%%%%%%%%%%%%%%%%%%%%%%%%

\end{document}